\documentclass[12pt]{article} 
\usepackage[utf8]{inputenc}
\usepackage[T2A]{fontenc}\usepackage[english]{babel}
\usepackage{amsthm, amsmath, amssymb, amsbsy, amscd, amsfonts, latexsym, euscript,epsf,stackrel}
\usepackage{authblk}
\usepackage{graphicx} 

\usepackage{epsfig}
\newtheorem{theorem}{Theorem}[section]
\newtheorem{corollary}[theorem]{Corollary}
\newtheorem{lemma}[theorem]{Lemma}
\newtheorem{proposition}[theorem]{Proposition}

\newtheorem{question}[theorem]{Question}

\newtheorem{definition}[theorem]{Definition}
\newtheorem{example}[theorem]{Example}

\numberwithin{equation}{subsection}
\newtheorem*{ack}{Acknowledgement}
\usepackage[all,cmtip]{xy}
\usepackage{xcolor}
\usepackage{xfrac}
\usepackage{pb-diagram}

\newcommand{\Alex}{\operatorname{Alex}}
\newcommand{\Aut}{\operatorname{Aut}}

\newcommand{\Conj}{\operatorname{Conj}}
\newcommand{\Core}{\operatorname{Core}}

\newcommand{\R}{\operatorname{R}}

\newcommand{\GL}{\operatorname{GL}}

\def\bea{\begin{eqnarray}}
\def\eea{\end{eqnarray}}

\begin{document}
\title{Multi-groups}

\author{Tatyana A. Kozlovskaya\footnote{t.kozlovskaya@math.tsu.ru}}

\affil{\small Regional Scientific and Educational Mathematical Center of Tomsk State University,
36 Lenin Ave., 634050, Tomsk, Russia.  
}

%\subjclass[2010]{Primary 20N20; Secondary  Secondary 16S34, 05E30}
%\keywords{Multi-set, multivalued group,  $n$-valued dynamic, growth function}

\maketitle

\begin{abstract}
In the present paper we define  homogeneous algebraic systems. 
Particular cases of these systems are: semigroup (monoid, group) system. These algebraic systems were studied by  J.~Loday, A.~Zhuchok, T.~Pirashvili, N.~Koreshkov.
Quandle systems were introduced and studied by V.~Bardakov, D.~Fedoseev, V.~Turaev.

We construct  some group systems on the set of square matrices over a field $\Bbbk$. Also we define rack system  on the set $V \times G$, where $V$ is a vector space of dimension 
$n$ over $\Bbbk$, $G$ is a subgroup of $\GL_n(\Bbbk)$.Finally, we find  connection between skew braces and dimonoids.

\end{abstract}



\tableofcontents

\section{Introduction}

In the theory of algebraic systems there exist algebraic systems with  a set of one type algebraic operations. Let us give some examples of these algebraic systems.

Brace  (skew brace)  is a set with two group operations, which satisfy some axiom (see \cite{Rump}-\cite{GV}). A generalization of skew braces was suggested in the paper Bardakov-Neshchadim-Yadav \cite{BNY}, where were introduced brace systems as a set with a family of group operations which connected by some axioms.

Dimonoids were introduced by J.-L.~Loday \cite{Loday2} in his construction of a universal enveloping algebra for the Leibniz algebra. Dimonoid is a set with two semigroup operations, which are connected by a set of axioms. The construction of a free dimonoid generated by a given set was presented in \cite{Loday2} and applied to the study of free dialgebras and a cohomology of dialgebras. Structural properties of free dimonoids have been investigated by A.~V.~Zhuchok in \cite{Zhuchok1}.
In \cite{Zhuchok2} was presented a construction of a free product of arbitrary dimonoids which generalizes a free dimonoid and describe its structure. 
Dimonoids are examples of duplexes, which were introduced by T.~Pirashvili in \cite{Pir}. A duplex is an algebraic system with two associative binary operations (without added connections between these operations). %The set of all permutations, gives an example of a duplex which is not a dimonoid.
T.~Pirashvili  constructed a free duplex generated by a given set via planar trees and proved that the set of all permutations forms a free duplex on an explicitly described set of generators.

In \cite{Kor}, N.~Koreshkov introduced $n$-tuple semigroup as  an algebraic system 
$$
\mathcal{S} = (S, *_i, i \in I)
$$
 such that $(S, *_i)$ is a semigroup for any $i \in I$
and with the following axiom which connects these operations,
$$
(a *_i b) *_j c = a *_i (b *_j c),~~ a, b, c \in S,~~i, j \in I. 
$$
The free $n$-tuple semigroup of an arbitrary rank was first constructed in  \cite{Zhuchok3}.

In the present paper we define  homogeneous algebraic systems (see Definition \ref{HAS}). Particular cases of these systems are: semigroup (monoid, group) system $\mathcal{G} = (G, *_i, i \in I)$, where $(G, *_i)$ is a semigroup (monoid, group) for any $i \in I$. An example of a semigroup system with two operations is a duplex. We will call  $\mathcal{G}$ by multi-semigroup (multi-monoid, multi-group) if the operations are  connected by the following condition
$$
(a *_i b) *_j c = a *_i (b *_j c),~~ a, b, c \in G,~~i, j \in I. 
$$
An example of a multi-semigroup with $n$ operations is an $n$-tuple semigroup \cite{Kor}.

V. G. Bardakov and D. A.~Fedoseev considered quandle system $\mathcal{Q} = (Q, *_i, i \in I)$ in \cite{BF}, where  $(Q, *_i)$ is a quandle for any $i \in I$, and defined a multiplication $*_i *_j$ of the operations $*_i$ and $*_j$ by the rule
$$
p (*_i *_j) q = (p *_i q) *_j q,~~p, q \in Q.
$$
In the general case the algebraic system $(Q, *_i *_j)$ is not a quandle, but if the operations are satisfied the axioms
\begin{equation} \label{quand}
(x *_i y) *_j z= (x *_j z) *_i (y *_j z),~~(x *_j y) *_i z= (x *_i z) *_j (y *_i z),~~~x, y, z \in Q,
\end{equation}
then  $(Q, *_i *_j)$ and  $(Q, *_j *_i)$ are quandles.
 V.~Turaev called quandle systems, which satisfy axioms (\ref{quand}) for all $i, j \in I$ by multi-quandles and gave them a  topological interpretation (see \cite{T}).

In 1971, V.~M.~Buchstaber and S.~P.~Novikov \cite{BN} introduced a notion of $n$-valued group,  in which the product of each pair of elements is an $n$-multi-set, the set of $n$ elements with multiplicities.  An appropriate survey on $n$-valued groups and its applications  can be found in \cite{B}.

If we have a group system $\mathcal{G} = (G, *_i, i \in I)$, where $|I| = n$, we can can define $n$-valued multiplication 
$$
a * b = [a *_1 b, a *_2 b, \ldots,  a *_n b],~~a, b \in G,
$$ 
and study the algebraic system $(G, *)$. In \cite{BKT} were investigated connections between group systems and $n$-valued groups and proved that if all groups 
$(G, *_i)$ have the general unit and $(G, *)$ is an $n$-valued group, then $*_i = *_j$ for all $1 \leq i, j, \leq n$.

In the present paper we are studying connections between skew braces and dimonoids, define a semigroup systems on the set of square matrices. Also, we investigate semigroup systems on the set of matrices $M_n(\Bbbk)$ and give an answer on a question from \cite{BKT}. Also, we construct some rack systems and multi-racks on the set 
$V \times G$, where $V$ is a vector space of dimension $n$ over a field $\Bbbk$, $G$ is a subgroup of $\GL_n(\Bbbk)$.

The paper is organized as follows.

In Section \ref{prelim1} we introduce definition of algebraic systems, which include the algebraic systems from the introduction. 

In Section \ref{sec-prelim} we construct some group systems on the set of square matrices over a field and give an answer on a question from \cite{BKT}.

In Section \ref{R} will be define rack system  on the set $V \times G$, $G \leq  \GL_n(\Bbbk)$.

In Section \ref{dim} will be established connection between skew braces and dimonoids.

\bigskip

%%%%%%%%%%%%%%%%%%%%%%%%%%%%%%%%%%%%%%%%%%%%%%%%%%%%5

\section{Homogeneous algebraic systems } \label{prelim1}

In this section we introduce definition of algebraic systems, which include the algebraic systems from the introduction.  

\begin{definition}  \label{HAS}
Let $\mathcal{A} = (A, f_i, i \in I)$ be an algebraic system with a set of algebraic operations $f_i$ of arity $n_i$. It is said to be $m$-homogeneous $I$-system if all arities $n_i$ are equal to $m$. In particular, if $|I|=n$, we will say on $n$-system. If $m = 2$ we will say instead $2$-homogeneous $I$-system on groupoid  $I$-system or simply on groupoid system.   
\end{definition} 

The typical example is a ring $(K, +, \cdot)$ that is groupoid 2-system.
Other examples of $2$-homogeneous $I$-system are: semigroup (monoid, group) system $\mathcal{G} = (G, *_i, i \in I)$, where $(G, *_i)$ is a semigroup (monoid, group) for any $i \in I$. An example of a semigroup system with two operations is a duplex. We will call a system $\mathcal{G}$ by multi-semigroup (multi-monoid, multi-group) if the operations are  connected by the following condition
$$
(a *_i b) *_j c = a *_i (b *_j c),~~ a, b, c \in S,~~i, j \in I. 
$$
An example of a multi-semigroup with $n$ operations is an $n$-tuple semigroup (see the introduction).

Let us give other examples of semigroup systems. 

{\it Skew braces} (see \cite{Rump},  \cite{GV}).
A  triple $(G, \cdot, \circ)$, where $(G, \cdot)$ and $(G, \circ)$ are  groups,    is said to be a \emph{skew (left) brace} if
 \begin{equation}
 g_1 \circ (g_2 \cdot g_3) =  (g_1 \circ g_2) \cdot g_1^{-1} \cdot  (g_1 \circ g_3)
 \end{equation}
 for all $g_1, g_2, g_3 \in G$, where $ g_1^{-1}$ denotes the  inverse of $g_1$ in $(G, \cdot)$. We call  $(G, \cdot)$ the \emph{additive group} and $(G, \circ)$ the \emph{multiplicative  group} of the skew left brace $(G, \cdot, \circ)$. A skew left brace $(G, \cdot, \circ)$ is said to be a \emph{(left) brace} if $(G, \cdot)$ is an abelian group. In this case we will use the notation $+$ instead $\cdot$ in additive group. We see that a skew left brace is an example of group system with 2 operations.

{\it Dimonoids} (see \cite{Loday2},  \cite{Loday1}).
A  dimonoid  is a set X together with two binary operations $\vdash$ and $\dashv$ satisfying the following axioms:

$$\begin{cases} x \dashv  (y \dashv  z) \overset{1}{=} (x \dashv  y) \dashv  z\overset{2}{=}x \dashv (y \vdash z), &\\ (x \vdash y) \dashv  z \overset{3}{=} x \vdash (y \dashv z),&\\ (x \dashv y) \vdash z \overset{4}{=} x \vdash (y \vdash z) \overset{5}{=} (x \vdash y) \vdash z.&\end{cases} $$
for all $x,y$ and $z \in X$.
Observe that relations $1$ and $5$ are the ``associativity''  of the products $\vdash$
and $\dashv$ respectively. 
 The typical examples of  dimonoid are the following. 

a) Let $M$ be a monoid. Put $D = M \times M$ and define the products by 

$$(m, n) \dashv (m^{\prime}, n^{\prime}):= (m, nm^{\prime} n^{\prime}), $$
$$(m, n)  \vdash (m^{\prime}, n^{\prime}) := (mnm^{\prime}, n^{\prime}).$$
With these definitions $D = (D, \dashv, \vdash)$ is a dimonoid. Let us check relation 3 for
instance:

$$((m,n) \vdash  (m^{\prime}, n^{\prime})) \dashv (m^{\prime \prime}, n^{\prime \prime}) = (mnm^{\prime}, n^{\prime} ) \dashv (m^{\prime \prime}, n^{\prime \prime}) = (mnm^{\prime}, n^{\prime} m^{\prime \prime} n^{\prime \prime})$$
$$ (m,n) \vdash  ((m^{\prime}, n^{\prime}) \dashv (m^{\prime \prime}, n^{\prime \prime})) =  (m,n) \vdash (m^{\prime}, n^{\prime} m^{\prime \prime} n^{\prime \prime}) = (mnm^{\prime}, n^{\prime} m^{\prime \prime} n^{\prime \prime}).$$

Let $1 \in M$  be a unit for $M$. Then $e=(1,1)$ is a bar-unit for $D$, but one has $e \dashv x \neq x$ and $x \vdash x  \neq x$ in $D$ in general. For any invertible element $m$ the element $(m, m^{-1} ) \in D$ is a bar-unit.

b) Let $G$ be a group and $X$ be a $G$-set. The following formulas define a dimonoid structure on $X \times G$:

$$ (x,g) \dashv (y,h):=(x, gh),$$
$$ (x,g) \vdash (y,h):=(g\cdot x, gh).$$

We see that a dimonoid is an example of group system  with 2 operations.

\bigskip

%%%%%%%%%%%%%%%%%%%%%%%%%%%%%%%%%%%%%%%%%%%

\section{Group systems and multi-groups} \label{sec-prelim}

Let $M_n(\Bbbk)$ be a set of $n \times n$ matrices  over a field $\Bbbk$. 
The next multiplication was defined in \cite{BKT}:
$$
A *_{s,t,M_1, M_2} B = s A M_1 B + t A M_2 B,~~~s, t \in \Bbbk,~~M_1, M_2 \in M_n(\Bbbk),
$$
and was formulated

\begin{question}
What can we say on this product? What algebraic systems one can construct using these  operations? Is there connection of these operations with non-standard matrix 
multiplications which were  studied in \cite{BS}?
\end{question}

Let us find conditions under which $(M_{n}(\Bbbk), *_{s,t,M_1, M_2} )$ is a semigroup. It is need to check axiom of associativity,
$$(A*B)*C= A*(B*C),~~A, B, C \in M_{n}(\Bbbk).$$
We have
$$(A*B)*C= (s A M_1 B + t A M_2 B)*C=s (s A M_1 B + t A M_2 B) M_1 C + t (s A M_1 B + t A M_2 B) M_2 C,$$
on the other side,
 $$A*(B*C)=A*(s B M_1 C + t B M_2 C)= s A M_1 ( s B M_1 C + t B M_2 C)  + tAM_2 (s B M_1 C + t B M_2 C).$$
We have a system
$$\begin{cases} AM_1 B M_1 C= AM_1BM_1C;  &\\ AM_2BM_1C+AM_1BM_2C=AM_1BM_2C+AM_2BM_1C;&\\ AM_2BM_2C=AM_2BM_2C.&\end{cases} $$
It is easy to see that the associativity axiom holds. Hence, we have proven the next lemma.

\begin{lemma}
The operation $*_{s,t,M_1, M_2}$ is associative.
\end{lemma}

\begin{corollary}
The algebraic system $(M_n(\Bbbk), *_{s,t,M_1, M_2}, s, t \in \Bbbk, M_1, M_2 \in M_n(\Bbbk))$ is a semigroup system.
\end{corollary}

Let us check, is this semigroup system a multi-semigroup? To do it it is need to take 
 two different operations: $ *= *_{s,t,M_1, M_2}$, $ \circ= \circ_{p,q,N_1, N_2}$ and check the  axiom
$$
(A*B)\circ C=A*(B\circ C).
$$
The left hand side:
$$(A*B)\circ C = (s A M_1 B + t A M_2 B)\circ C =  p (s A M_1 B + t A M_2 B) N _1 C + q (s A M_1 B + t A M_2 B) N_2 C.$$
The right hand side:
 $$A*(B\circ C)= A*(p B N_1 C + q B N_2 C)= s A M_1 ( p B N_1 C + q B N_2 C)  + tAM_2 (p B N_1 C + q B N_2 C).$$
We get a system

$$
\begin{cases} AM_1 B N_1 C= AM_1BN_1C;  &\\  AM_2 B N_1 C= AM_2BN_1C;  &\\  AM_1 B N_2 C= AM_1BN_2C; &\\  AM_2 B N_2 C= AM_2BN_2C. &\end{cases} 
$$
Since, this system true for all matrices, we get

\begin{proposition}
The semigroup system $(M_n(\Bbbk), *_{s,t,M_1, M_2}, s, t \in \Bbbk, M_1, M_2 \in M_n(\Bbbk))$  is a multi-semigroup.
\end{proposition}

Let us find the unit element:
$$
A*X=s A M_1 X + t A M_2 X=A.
$$
 It means that $t=0, s=1, M_1X=E \Rightarrow  X= M_1^{-1}.$

Hence,  $A*X= AM_1X,  E^{(*)}=M_1^{-1}$.
On the other side $X*A=XM_1A=A$.

\begin{lemma}
We will have the unit element only for multiplication
 $A*B=AMB$, ${det M}{\neq }0$, $ E^{(*)}=M^{-1}.$
\end{lemma}

Inverse element
$A*Y=E^{(*)}\Leftrightarrow  AMY=M^{-1}.$ Hence, $Y=M^{-1}A^{-1}M^{-1}.$

\begin{theorem}
1) Let $M \in M_n(\Bbbk)$,  ${det M}{\neq }0$. Then $(GL_n(\Bbbk), *_M)$ is a group with the product $A *_M B = A M B$, with the unit element $E^{(*)}=M^{-1}$ and the inverse  $\bar{A}^{(*)}=M^{-1}A^{-1}M^{-1}.$

2) The algebraic system $(GL_n(\Bbbk), *_M, M\in GL_n(\Bbbk))$ is a multi-group.
\end{theorem}

\bigskip

%%%%%%%%%%%%%%%%%%%%%%%%%%%%%%%%%%%%%%%%%%%%%%%%%

\section{Rack systems } \label{R}

Some examples of quandle systems and multi-quandles can be found in \cite{BF, T}. In this section we give some other examples. At first, recall basic definitions.

\begin{definition}[{\cite{M}, \cite{J}}] 
A {\it quandle} is a non-empty set $Q$ with a binary operation $(x,y) \mapsto x * y$ satisfying the following axioms:
\begin{enumerate}
\item[(Q1)] $x*x=x$ for all $x \in Q$,
\item[(Q2)] for any $x,y \in Q$ there exists a unique $z \in Q$ such that $x=z*y$,
\item[(Q3)] $(x*y)*z=(x*z) * (y*z)$ for all $x,y,z \in Q$.
\end{enumerate}

An algebraic system satisfying only (Q2) and (Q3)  is called a {\it rack}. Many interesting examples of quandles come from groups.
\end{definition}
\par

\begin{example}
\begin{enumerate}
\item If $G$ is a group, $m$ is an integer,  then the binary operation $a*_m b= b^{-m} \, a \,  b^m$ turns $G$ into a quandle $\Conj_m(G)$ called the $m$-{\it conjugation quandle} on~$G$. 
If $m=1$, this quandle is called {\it conjugation quandle} and is denoted $\Conj(G)$.
\item A group $G$ with the binary operation $a*b= b a^{-1} b$ turns the set $G$ into a quandle $\Core(G)$ called the {\it core quandle} of $G$. In particular, if $G= \mathbb{Z}_n$ is  the cyclic group of order $n$, then it is called the {\it dihedral quandle} and denoted by $\R_n$.
\item Let $G$ be a group and $\varphi \in \Aut(G)$. Then the set $G$ with binary operation $a *_{\varphi} b = \varphi(ab^{-1})b$ forms a quandle $\Alex(G,\ \varphi)$ referred as the  {\it generalized Alexander quandle} of $G$ with respect to $\varphi$.
\end{enumerate}
\end{example}

From the last example follows that  if $Q = \GL_n(\Bbbk)$, $\varphi \in \Aut(\GL_n(\Bbbk))$, then we can define a quandle system $(Q, *_{\varphi}, \varphi \in \Aut(GL_n(\Bbbk)))$.

In the present section we are studying the next question: What rack (quandle) systems one can defines on $V \times G$, where $V$ is a vector space of dimension $n$ over a field $\Bbbk$, $G$ is a subgroup of $\GL_n(\Bbbk)$.
On the set $Q=V \times G$ we
can define the operation
$$
(a, A)\circ (b, B)=(A b, \varphi(AB^{-1})B),~~~a, b \in V,~~A, B \in G,~~\varphi \in \Aut(G).
$$
In this case 
$$(a, A) \circ (a, A)=(A a, \varphi(E)A)=(A a, A).$$
It means that $A = E$, hence, $G = \{ E \}$ is the trivial group.

The second quandle  axiom:

$$
(u,X)\circ (a, A)=(b,B) \Leftrightarrow (u,X)\circ (a, A)=(Xa,\varphi(XA^{-1})A).
$$

Hence, 
$$
Xa = b,~~X=\varphi^{-1}(BA^{-1})A.
$$
It means that such element $(u, X)$ exists, but it is  not unique. 

Let us check the  third axiom:
$$
((a,A)\circ (b, B)) \circ (c, C)=((a,A)\circ (c, C)) \circ ((b,B) \circ (c, C)).
$$
The left hand side,
$$
((a,A)\circ (b, B)) \circ (c, C)=(Ab, \varphi(AB^{-1})B) \circ (c, C) =(\varphi (AB^{-1})Bc, \varphi(\varphi (AB^{-1})BC^{-1})C).
$$
The right hand  side
$$((a,A)\circ (c, C)) \circ ((b,B) \circ (c, C))=(Ac, \varphi(AC^{-1})C) \circ (Bc, \varphi(BC^{-1})C) =$$
$$= (\varphi (AC^{-1})CBc, \varphi(\varphi (AC^{-1})C).$$

We  have the system

$$\begin{cases} \varphi (AB^{-1})Bc= \varphi (AC^{-1})CBc;&\\ \varphi(\varphi (AB^{-1})BC^{-1})C= \varphi(\varphi (AC^{-1})C). &\end{cases} $$
Since, $A * B = \varphi (AB^{-1})B$ defines a quandle operation, the second equation is true.
Consider the first equation of the system. It equivalents to the equality
$$
\varphi (A B^{-1} CA^{-1}) = C,
$$
which must be true for arbitrary $A, B, C \in G$. Evidently, that for trivial group it is true.

Let us define the operation (see \cite{kinyon}):
$$
(a, A)\circ (b, B)=(A b, A B A^{-1}),~~~a, b \in V,~~A, B \in G, 
$$
and check the left self-distributivity,
$$
(a,A)\circ ((b, B) \circ (c, C))=((a,A)\circ (b, B)) \circ ((a,A) \circ (c, C)).
$$
Since
$$(a,A)\circ ((b, B) \circ (c, C))= (a, A)\circ (Bc, BCB^{-1})=(ABc, ABCB^{-1}A^{-1}) $$
and
$$((a,A)\circ (b, B)) \circ ((a,A) \circ (c, C))= (Ab,ABA^{-1}) \circ (Ac,ACA^{-1})=(ABc, ABCB^{-1}A^{-1}), $$
then the  left self-distributivity holds.

Let us take $n \in \mathbb{Z}$ and define more general operation,
$$
(a, A)\circ_n (b, B)=(A^n b, A^n B A^{-n}),~~~a, b \in V,~~A, B \in G.
$$
Check the left self-distributivity,
$$(a,A) \circ_n ((b, B)) \circ_n ((c,C))=(a,A) \circ_n (B^{n}c, B^{n}CB^{-n})=(A^{n}B^{n}c, A^{n}B^{n}CB^{-n}A^{-n}),$$
$$((a,A) \circ_n (b, B)) \circ_n ((a,A) \circ_n (c, C))=(A^{n}b,A^{n}BA^{-n})\circ_n  (A^{n}c,A^{n}CA^{-n})=(A^{n}B^{n}c, A^{n}B^{n}CB^{-n}A^{-n}).$$

Hence, the operation 
$(a,A)\circ_n (b, B)=(A^{n}b,A^{n}BA^{-n})$  is left self-distributive.

Let us check the left divisibility axiom: $(a,A)\circ_n (u, X)=(b,B)$
We have 
$$
(a,A)\circ_n (u, X)=(A^{n}u,A^{n}XA^{-n}).
$$ Hence, 
$$\begin{cases} u=A^{-n}b; &\\ X= A^{-n}BA^{n}.&\end{cases} $$
Since this system has unique solution, the left divisibility holds.

Summarize the previous calculations, we get

\begin{theorem}
Let $Q = (V, G)$, where $V$ is a vector space of dimension $n$ over a field $\Bbbk$, $G$ is a subgroup of $\GL_n(\Bbbk)$. Then the algebraic system
$(Q, *_n, n \in \mathbb{Z})$, where
$$
(a, A)\circ_n (b, B)=(A^n b, A^n B A^{-n}),~~~a, b \in V,~~A, B \in G
$$
satisfies the axioms:

1) Left self-distributivity,
$$
(a,A)\circ_n ((b, B) \circ_n (c, C))=((a,A)\circ_n (b, B)) \circ_n ((a,A) \circ_n (c, C)), ~~a, b, c \in V,~~A, B, C \in G.
$$

2) Left divisibility,

for any $(a, A)$, $(b, B) \in Q$ there is unique $(u, X) \in Q$ such that 
$$
(a, A) \circ_n (u, X) = (b, B). 
$$
\end{theorem}

From this theorem follows

\begin{corollary}
The algebraic system $(Q, *^{op}_n, n \in \mathbb{Z})$, where the  opposite operations are defined by the rules
$$
(a, A) *^{op}_n (b, B) = (b, B) *_n (a, A) 
$$
is a rack system.
\end{corollary}

\bigskip

%%%%%%%%%%%%%%%%%%%%%%%%%%%%%%%%%%%%%%%%%%%%%%%%

\section{Connection between skew braces and dimonoids} \label{dim}

In this section we find some connections between skew braces and dimonoids.

\begin{proposition}
Let $(G, \cdot)$ be a group. 

1) If $a\circ b=a\cdot b$ then we get a skew brace. If $a \vdash b = a \dashv b=ab$ then we get a dimonoid.

2) If $a\circ b=ba$, then $(G,\cdot,\circ)$ is  skew brace. If   $a \vdash b =ab$ and $a \dashv b=ba$ then  $(G, \dashv , \vdash )$ is  not a dimonoid.

\end{proposition}

\begin{proof}

The  binary operation $ \vdash$ is associative since it corresponds to the product in the group~$G$.

Let us check the following axiom:
$$ (a \dashv b) \dashv  c =a  \dashv (b \dashv c).$$
To do it, we have to  compute both sides of the equation:
$$(ba)  \dashv c=c(ba),$$
$$a \dashv (cb)=(cb)a.$$
Since they are the same, the operation is associative.

Let us check the following axiom:
$a \dashv  (b \dashv  c) =a \dashv (b \vdash c)$.
We have $cba=bca$. Therefore, it must satisfy $bc=cb$ and this group is  an abelian group.
It means that if $a \dashv b= a \vdash b = ab$, then $(G,\vdash,\dashv)$ is a dimonoid.

\end{proof}

If we a have a  skew brace $(G,\cdot,\circ)$, then we can define operations  
 $a \vdash b = ab,  a \dashv b = a\circ b$ and formulate the question: is $(G,\vdash, \dashv)$ a dimonoid?

The next example shows that in the general case the answer is negative.

\begin{example}
Let us take the brace   $(\mathbb{Z},+,\circ)$, where $(\mathbb{Z},+)$ is the infinite cyclic group and $a\circ b=a+(-1)^{a}b$, $a, b \in \mathbb{Z}$.
Note that 
$$a\circ b=a+(-1)^{a}b=\begin{cases}a+b,&if\  a\  is\  even; \\ a-b,&  if\  a\  is\ odd. \end{cases} 
$$

Put
$$
a \vdash b = a+b,  a \dashv b = a\circ b.
$$

It is evident that the associativity holds for the binary operations $\vdash$ and $\dashv$.
Let us check that 
$$a \dashv (b \dashv c)=a \dashv (b \vdash c). $$

We have that 
$$a \circ (b \circ c)= a \circ \begin{cases} b+c,&if\  b\  is\  even; \\ b-c,&  if\  b\  is\ odd.\end{cases} =  \begin{cases}a+b+c&if \ a,b\  are\  even; \\ a-b-c& if  \ a \  is\ odd \ and \ b\  is\  even; \\ a+b-c&  if  \ a \  is\ even \ and \ b\  is\  odd ;  \\ a-b+c& 
 if \ a,b\  are\  odd.\end{cases} $$

On the other side we get
$$
a \circ (b + c)= \begin{cases} a+b+c,&if\  a\  is\  even; \\ a-b-c,&  if\  a\  is\ odd.\end{cases} 
$$
If we  take $a = 2$, $b = 3$, $c = 4$,  then $a \dashv (b \dashv c)= a+b-c=1. $ On the other side
$a \dashv (b \vdash c) =  a+b+c =9. $

Therefore, the skew brace  $(\mathbb{Z},+,\circ)$ is not a dimonoid.

\end{example}

At the end we formulate the following question.

\begin{question}
Under which conditions a skew brace $(G,\cdot,\circ)$ is a dimonoid with respect to the operations
 $a \vdash b = ab,  a \dashv b = a\circ b?$
\end{question}

\begin{question}
Let $\mathcal{G} = (G, *_i, i \in I)$ be a semigroup system. Define a product of semigroup operations,
$$
g (*_i *_j) h = (g *_i h) *_j h,~~g, h \in G.
$$
Find necessary and sufficient conditions under which $(Q, *_i *_j)$ is a semigroup.

\end{question}

%%%%%%%%%%%%%%%%%%%%%%%%%%%%%%%%%%%%%%

%\end{document}

\begin{ack}
This work was supported by the Ministry of Science and Higher Education of Russia (agreement No. 075-02-2023-943).
\end{ack}

\end{document}